\documentclass[11pt]{article}

\usepackage[
    top=27mm,
    bottom=27mm,
    left=22mm,
    right=22mm
]{geometry}

\usepackage[square,comma,sort&compress,numbers]{natbib}

\usepackage{amsmath,amssymb,amsthm,mathrsfs}

\usepackage[normalem]{ulem}
\usepackage{xcolor}
\usepackage{hyperref}
\usepackage[capitalize]{cleveref}

\newtheorem{theorem}{Theorem}[section]
\newtheorem{lemma}[theorem]{Lemma}
\newtheorem{corollary}[theorem]{Corollary}
\newtheorem{claim}[theorem]{Claim}

\theoremstyle{definition}
\newtheorem{definition}[theorem]{Definition}
\newtheorem{construction}[theorem]{Construction}
\theoremstyle{plain}

\newcommand{\fr}{\frac}
\newcommand{\lc}{\lceil}
\newcommand{\rc}{\rceil}
\newcommand{\lf}{\lfloor}
\newcommand{\rf}{\rfloor}

\newcommand{\blue}{\textcolor{blue}}

\begin{document}

\title{Extremal problems for cancellative and locally thin hypergraphs}

\date{}

\makeatletter
\def\thanks#1{\protected@xdef\@thanks{\@thanks
        \protect\footnotetext{#1}}}
\makeatother

\author{Miao Liu\thanks{M. Liu is with the Research Center for Mathematics and Interdisciplinary Sciences, Shandong University, Qingdao, China. Email: liumiao10300403@163.com.}, Chong Shangguan\thanks{C. Shangguan is with the Research Center for Mathematics and Interdisciplinary Sciences, Shandong University, Qingdao 266237, China, and the Frontiers Science Center for Nonlinear Expectations, Ministry of Education, Qingdao 266237, China. Email: theoreming@163.com.}, 
and Chenyang Zhang \thanks{C. Zhang is with the Research Center for Mathematics and Interdisciplinary Sciences, Shandong University, Qingdao 266237, China. Email: lxzcyang@163.com.}
}

\maketitle

\begin{abstract}
We study Tur\'an-type extremal problems for cancellative and locally thin uniform hypergraphs. An $r$-uniform hypergraph is $t$-cancellative if $(\cup_{i=1}^t A_i)\cup B\ne (\cup_{i=1}^t A_i)\cup C$ whenever $A_1,\ldots,A_t,B,C$ are distinct edges. Let $C_t(n,r)$ denote the maximum number of edges in such a hypergraph on $n$ vertices. For all fixed integers $t,k\ge2$, we prove that
$C_{2(t-1)}(n,tk)=(1+o(1))\frac{\binom{n}{k}}{\binom{tk-1}{k-1}}$ as $n\to\infty$.
In the case $t=2$, this shows that F\"uredi's 2012 upper bound for $C_2(n,2k)$ is asymptotically sharp. The lower bound uses locally sparse induced packings, while the upper bound follows from double counting and a matching argument. 
   
More generally, for integers $s\ge t\ge1$, an $r$-uniform hypergraph is locally $(s,t)$-thin if among any $s$ distinct edges, at least $t$ contain a vertex that lies in none of the other $s-1$ edges. This notion includes cancellative hypergraphs as special cases. We establish general upper and lower bounds for the corresponding extremal numbers and determine their polynomial order of growth under suitable divisibility assumptions.
\end{abstract}

\section{Introduction}

An \emph{\(r\)-uniform hypergraph}, or simply an \emph{\(r\)-graph}, is a family of \(r\)-subsets of a finite vertex set. Given a family \(\mathscr L\) of \(r\)-graphs, let \(\mathrm{ex}_r(n,\mathscr L)\) denote the maximum number of edges in an \(n\)-vertex \(r\)-graph containing no copy of any member of \(\mathscr L\). Determining the asymptotic behaviour of such extremal functions is a central problem in extremal combinatorics \cite{keevash2011hypergraph}.

A basic problem is to determine the polynomial order of growth of this extremal function. For instance, one may seek an exponent satisfying \(\mathrm{ex}_r(n,\mathscr L)=\Theta(n^\alpha)\). A sharper problem is to determine whether there exists a positive constant satisfying \(\mathrm{ex}_r(n,\mathscr L)=(c+o(1))n^\alpha\).
Thus, even when the order of growth is known, the existence and the value of a leading asymptotic constant remain separate questions.

For example, it is known that \(\mathrm{ex}_2(n,\{C_3,C_4\})=\Theta(n^{3/2})\), while a conjecture of Erd\H{o}s from 1975 \cite{erdos1975some} asserts that \(\lim_{n\to\infty}\frac{\mathrm{ex}_2(n,\{C_3,C_4\})}{n^{3/2}}=\frac{1}{2\sqrt{2}}\), where \(C_\ell\) denotes a cycle of length \(\ell\). The corresponding questions for bipartite graphs are surveyed in \cite{Furedi-Simonovits-history}.
One standard hypergraph example is the \((v,e)\)-free problem. An \(r\)-graph is \emph{\((v,e)\)-free} if the union of every \(e\) distinct edges contains at least \(v+1\) vertices. For results on these problems, see, for example, \cite{Ruzsa-Szemeredi-63,Glock-triple,Shangguan-Tamo-degenerate,Furedi-Gerbner-without-exponents,Shangguan-degenerate-II,bes,Glock64,Letzter-Sgueglia-23,Glock-Kim-Lichev-Pikhurko-Sun-bes,Pikhurko-Sun-8,Wang-Zeng-even-BES}. 

In this paper, we study extremal problems for two closely related classes of hypergraphs: cancellative hypergraphs and locally thin hypergraphs. For integers $r\ge2$ and $s\ge t\ge1$, an $r$-graph $\mathcal H$ is called 
\begin{itemize}
    \item \emph{$t$-cancellative} if $(\bigcup_{i=1}^t A_i)\cup B\ne(\bigcup_{i=1}^t A_i)\cup C$ for any $t+2$ distinct edges $A_1,\ldots,A_t,B,C\in\mathcal H$;
    \item \emph{locally $(s,t)$-thin} if, among any $s$ distinct edges, at least $t$ contain a vertex that belongs to none of the other $s-1$ edges.
\end{itemize}
These notions are closely related: $t$-cancellative hypergraphs are precisely locally $(t+2,t+1)$-thin hypergraphs \cite{Furedi-2-canc}.

Let $C_t(n,r)$ and $L_{(s,t)}(n,r)$ denote the maximum number of edges in an $n$-vertex $t$-cancellative $r$-graph and in an $n$-vertex locally $(s,t)$-thin $r$-graph, respectively. We determine the leading asymptotic constant of $C_{2(t-1)}(n,tk)$ for all fixed integers $t,k\ge2$. We also establish general upper and lower bounds for $L_{(s,t)}(n,r)$ and determine its polynomial order of growth under suitable divisibility assumptions.

\subsection{Cancellative $r$-graphs}

The case \(t=1\) is the classical notion of cancellativity, introduced by Erd\H{o}s and Katona \cite{Erdos-Katona-1-canc}. When $r=2$, the $1$-cancellative condition is equivalent to being triangle-free. Mantel's theorem \cite{Mantel-K_3} gives $C_1(n,2)=\lfloor \frac{n^2}{4} \rfloor$. More generally, every complete $r$-partite $r$-graph is cancellative, so 
$$C_1(n,r)\ge p(n,r):=\Bigl\lfloor \frac{n}{r}\Bigr\rfloor
\Bigl\lfloor \frac{n+1}{r}\Bigr\rfloor \cdots
\Bigl\lfloor \frac{n+r-1}{r}\Bigr\rfloor,$$
where $p(n,r)$ is the number of edges of the balanced complete $r$-partite $r$-graph on $n$ vertices.

Katona \cite{Erdos-Katona-1-canc} posed the problem of determining the maximum size of a cancellative $3$-graph on $n$ vertices and conjectured that the extremal construction is the balanced complete tripartite $3$-graph. This conjecture was later proved by Bollob\'as \cite{bollobas1974three}. Bollob\'as further conjectured that $C_1(n,r)=p(n,r)$ for all $r\ge4$. This was proved for $r=4$ by Sidorenko \cite{sidorenko1987maximal} and for $r\le n\le2r$ by Frankl and F\"uredi \cite{Frankl-Furedi-1-canc-upp}, but later Shearer \cite{shearer1996new} provided counterexamples for sufficiently large $r$. Currently, the best known lower and upper bounds in \cite{tolhuizen-1-canc-low,Frankl-Furedi-1-canc-upp} give that
$$\fr{0.28}{2^r}\binom{n}{r}<C_1(n,r)\le\fr{2^r}{\binom{2r}{r}}\binom{n}{r}.$$

K\"orner and Sinaimeri \cite{Korner-Sinaimeri-2-canc} studied the case $t=2$. F\"uredi \cite{Furedi-2-canc} subsequently considered $t$-cancellative hypergraphs for general $t$; in particular, he showed that $C_2(n,4)=\frac{n^2}{6}+o(n^2)$, and more generally,
\begin{align}\label{eq:2-canc}
    \frac{n^k}{(2k)^k}-o(n^k)\le C_2(n,2k)\le\frac{\binom{n}{k}}{\binom{2k-1}{k-1}}.
\end{align}
\noindent It is also known that $n^{k-o(1)}<C_2(n,2k-1)=O(n^{k})$; see \cite{Shangguan-Tamo-canc-and-union-free,Furedi-2-canc}. 

Shangguan and Tamo \cite{Shangguan-Tamo-canc-and-union-free} obtained general lower and upper bounds for $C_t(n,r)$ and proved for all fixed integers $t,k\ge2$, $C_{2(t-1)}(n,tk)=\Theta(n^k)$. This leaves the leading asymptotic constant undetermined. Within this family of parameters, the constant had previously been determined only for $C_2(n,4)$, by F\"uredi \cite{Furedi-2-canc}. Our first theorem determines it for
all fixed $t,k\ge2$.

\begin{theorem}\label{thm:main-canc}
    For all fixed integers $t\ge 2$ and $k\ge 2$, we have
    \begin{align}\label{eq:canc}
        C_{2(t-1)}(n,tk)=(1+o(1))\frac{\binom{n}{k}}{\binom{tk-1}{k-1}}, \qquad \text{as}~n\to\infty,
    \end{align}
\end{theorem}

For $t=2$, this shows that the upper bound in \eqref{eq:2-canc} is asymptotically sharp for every $k\ge2$.

\paragraph{Proof ideas and relation to earlier work.} The proof of \cref{thm:main-canc} is divided into two parts. For the lower bound, we first construct a suitable $2(t-1)$-cancellative $tk$-graph $\mathcal{F}$, and then combine many copies of it on a packing of its $k$-shadow (see \cref{construction}). The main difficulty is to ensure that the resulting hypergraph remains $2(t-1)$-cancellative.  An ordinary packing does not guarantee this, since edges coming from several different copies may together form a forbidden configuration for cancellativity. We overcome this difficulty by using  an induced and locally sparse packing (see \cref{def:packing,lem:conflict-free-matching}). The induced packing condition controls pairwise interactions between copies, while local sparsity controls configurations involving several copies. These properties are combined in \cref{lem:canc-H(F)}, and the resulting construction matches the leading constant in the upper bound. 

For the upper bound, we distinguish between the $k$-sets contained in
exactly one edge and those contained in at least two edges. For a $k$-set
of the latter type, cancellativity imposes a matching-number restriction
in all but at most one of the edges containing it (see \cref{claim:uf-UB}). Kleitman's bound in \cref{lem:matching-conjecture} then yields sufficiently many $k$-sets of the former type. A double-counting argument gives the desired upper bound.

\paragraph{Related work.} An earlier version of \cref{thm:main-canc} appeared in the preprint~\cite{liu2024cancellative-union-free}, which considered both cancellative and union-free hypergraphs. As the results on both problems were subsequently expanded substantially, that preprint was divided into two separate papers. The union-free results appear in~\cite{liu2026union-free}. The present paper develops the cancellative part and, in addition, establishes new results for the broader class of locally $(s,t)$-thin uniform hypergraphs discussed in the next subsection. 

\subsection{Locally $(s,t)$-thin}

We next recall the relevant background on locally thin families. The case $(s,t)=(4,1)$ was studied by Alon, K\"orner, and Monti \cite{alon2000string}, and locally $(s,1)$-thin families for general $s$ were subsequently considered by Alon, Fachini, and K\"orner \cite{alon2000locally}. Most previous work concerns non-uniform set systems and their exponential growth rates \cite{Fachini-better-locally,Furedi-2-canc}.

In the uniform setting, locally thin families include two classical notions in extremal set theory: cancellative families and cover-free families. Following Erd\H{o}s, Frankl, and F\"uredi \cite{Erdos-Frankl-Furedi-2-cover-free,Erdos-Frankl-Furedi-r-cover-free}, a family is called \emph{$t$-cover-free} if for any distinct members $A_0,A_1,\ldots,A_t$, one has
$A_0 \nsubseteq \bigcup_{i=1}^t A_i$. F\"uredi \cite{Furedi-2-canc} observed that $t$-cancellative families are precisely locally $(t+2,t+1)$-thin families, while $t$-cover-free families are precisely locally $(t+1,t+1)$-thin families. Moreover, Frankl and F\"{u}redi \cite{Frankl-Furedi-colored-packing} determined the leading asymptotic constant for $t$-cover-free $r$-graphs, or equivalently, for locally $(t+1,t+1)$-thin $r$-graphs. In particular, for fixed $r$ and $s\ge2$,
$$L_{(s,s)}(n,r)=(\gamma(r,s-1)+o(1))n^{\lceil r/(s-1)\rceil},$$
for a positive constant $\gamma(r,s-1)$ determined in \cite{Frankl-Furedi-colored-packing}. 

There is an equivalent matrix formulation. The incidence matrix of an $r$-graph has one row for each vertex, one column for each edge, and constant column weight $r$. The $r$-graph is locally $(s,t)$-thin if and only if every submatrix formed by $s$ columns contains at least $t$ distinct rows of the $s\times s$ identity matrix. Thus locally $(s,t)$-thin $r$-graphs are precisely \emph{constant-column-weight selectors} with distinct columns. Fiore, Dalai, and Vaccaro \cite{DellaFiore-Dalai-Vaccaro-selectors} considered such selectors with an additional runlength constraint on each column. Under the correspondence $k=s$, $p=t$, and $w=r$, their Theorem~III.3 would imply the following bound for $s>t$ (after repeated columns are deleted):
$$L_{(s,t)}(n,r)=\Omega\left(n^{r(s-t+1)/(s-1)}\right).$$ 
However, this lower bound is false for a broad range of parameters.
For instance, when $(s,t)=(4,3)$ and $r$ is even, it predicts the
exponent $2r/3$, whereas locally $(4,3)$-thin $r$-graphs are precisely
$2$-cancellative $r$-graphs, whose exponent is known to be $r/2$;
see~\eqref{eq:2-canc}. Hence the predicted exponent is incorrect for
every even $r\ge4$. The error originates in equation~(17) of
\cite{DellaFiore-Dalai-Vaccaro-selectors}, which underestimates the
bad-event probability by ignoring the strong dependence among the bad
columns. 

We next give general bounds for $L_{(s,t)}(n,r)$ when $s>t$.

\begin{theorem}\label{thm:locally-thin-lower}
    Let $s>t\ge1$ and $r\ge2$ be integers. Then $L_{(s,t)}(n,r)=\Omega(n^h)$, where
    $$h=\max_{\substack{x\in\mathbb Z\\0\le x\le (s-t+1)r-1}}
    \min\left\{\frac{(s-t+1)r-x}{s-1},\frac{x+1}{s-t}\right\}.$$
\end{theorem}

\begin{theorem}\label{thm:locally-thin-upper}
    Let $s>t\ge1$ and $r\ge2$ be integers, and set $\ell=\left\lceil r/\left\lfloor (2s-t-1)/(s-t+1)\right\rfloor\right\rceil$. Then
    $$L_{(s,t)}(n,r)\le (s-1)\binom{n}{\ell}+s-1.$$
\end{theorem}

The following corollary is a direct consequence of Theorems \ref{thm:locally-thin-lower} and \ref{thm:locally-thin-upper}. 

\begin{corollary}\label{cor:locally-thin}
    Let $s>t\ge1$ and $r\ge2$ be integers, and set $d=(2s-t-1)/(s-t+1)$. If $d$ is an integer and $d\mid r$, then
    $$L_{(s,t)}(n,r)=\Theta\left(n^{r/d}\right)
    =\Theta\left(n^{(s-t+1)r/(2s-t-1)}\right).$$
\end{corollary}
Indeed, the assumptions give $\ell=r/d$ in Theorem~\ref{thm:locally-thin-upper}, while choosing $x=(s-t)r/d$ in Theorem~\ref{thm:locally-thin-lower} gives the matching lower exponent.

\paragraph{Organisation.} \Cref{sec:definitions-lemmas} contains the preliminary definitions and lemmas. The cancellative and locally thin results are proved in \cref{sec:thm:main-canc,sec:locally-thin}, respectively.

\section{Preliminaries}\label{sec:definitions-lemmas}

Unless stated otherwise, all parameters other than $n$ are fixed, and all asymptotic statements are taken as $n\to\infty$. We write $[n]:=\{1,\ldots,n\}$ and, for a finite set $V$, let $\binom{V}{r}$ denote the family of all $r$-subsets of $V$. For a hypergraph $\mathcal H$, let $V(\mathcal H)$ denote its vertex set. For sets $A$ and $B$, their \emph{symmetric difference} is $A\Delta B:=(A\setminus B)\cup(B\setminus A)$.

A collection of $e$ distinct edges whose union has size at most $v$ is called a \emph{$(v,e)$-configuration}.

We use the following standard probabilistic bound; see Lemma~1.7 of \cite{Furedi-Ruszinko-no-grid} and Proposition~6 of \cite{Shangguan-Tamo-spa-hyp-coding}.

\begin{lemma}\label{lem:e^--free}
    Let $r\ge2$ and $a\ge1$ be integers. For each $i\in[a]$, let $(v_i,e_i)$ be a pair of integers satisfying $v_i\ge r+1$ and $e_i\ge2$, and set $h:=\min_{i\in[a]}(e_ir-v_i)/(e_i-1)$. Then there exist constants $c>0$ and $n_0$ such that, for every integer $n\ge n_0$, there is an $n$-vertex $r$-graph with at least $cn^h$ edges that is $(v_i,e_i)$-free for every $i\in[a]$.
\end{lemma}

Since the union of two distinct $r$-edges has size at least $r+1$, every $(v,e)$-freeness condition with $v\le r$ is automatic. We omit such conditions when applying Lemma~\ref{lem:e^--free}.

We next introduce the packing terminology used in the proof of the lower bound in \cref{thm:main-canc}.

\begin{definition}[Packings]\label{def:packing}
    Let $\mathcal{J}$ be a fixed $k$-graph on $m$ vertices. A \emph{$\mathcal{J}$-packing} in $\binom{[n]}{k}$ is a family $\mathcal{P}=\{(V(\mathcal{J}_i),\mathcal{J}_i):i\in[|\mathcal{P}|]\}$ such that each $(V(\mathcal{J}_i),\mathcal{J}_i)$ is a copy of $\mathcal{J}$ and $\mathcal{J}_i\cap\mathcal{J}_{i'}=\emptyset$ for all distinct $i,i'$.

    The packing $\mathcal{P}$ is \emph{induced} if, for all distinct $i,i'$, we have $|V(\mathcal{J}_i)\cap V(\mathcal{J}_{i'})|\le k$, and whenever equality holds, $V(\mathcal{J}_i)\cap V(\mathcal{J}_{i'})$ is an edge of neither $\mathcal{J}_i$ nor $\mathcal{J}_{i'}$.

    For an integer $e\ge2$, the packing $\mathcal{P}$ is \emph{$e$-locally sparse} if the $m$-graph $\mathcal{V}(\mathcal{P}):=\{V(\mathcal{J}_i):i\in[|\mathcal{P}|]\}$ is $(\ell m-(\ell-1)k-1,\ell)$-free for every $2\le\ell\le e$.
\end{definition}

Since the copies in a $\mathcal{J}$-packing are pairwise edge-disjoint, every such packing satisfies $|\mathcal{P}|\le\binom{n}{k}/|\mathcal{J}|$. Results of R\"odl \cite{Rodl-nibble}, Frankl and R\"odl \cite{Frankl-Rodl-matching}, and Pippenger and Spencer \cite{Pippenger-Spencer-asymptotic} imply the existence of $\mathcal{J}$-packings of size $(1-o(1))\binom{n}{k}/|\mathcal{J}|$. Frankl and F\"uredi \cite{Frankl-Furedi-colored-packing} showed that such packings can also be chosen to be induced.

Using the conflict-free matching results of Delcourt and Postle \cite{Delcourtconflict} and Glock, Joos, Kim, K\"uhn, and Lichev \cite{Glockconflict}, We proved the following strengthening, which also imposes local sparsity \cite[Lemma~3.1]{liu2026union-free}. 

\begin{lemma}\label{lem:conflict-free-matching}
    Let $m>k$ and $e\ge2$ be fixed integers, and let $\mathcal{J}$ be a fixed $k$-graph on $m$ vertices with $|\mathcal{J}|\ge2$. For every $\epsilon>0$, there exists an integer $n_0$ such that, for all $n\ge n_0$, there exists a $\mathcal{J}$-packing $\mathcal{P}=\{(V(\mathcal{J}_i),\mathcal{J}_i):i\in[|\mathcal{P}|]\}$ in $\binom{[n]}{k}$ satisfying the following properties: 
    \begin{itemize}
        \item [(i)] $\mathcal{P}$ is induced; 
        \item [(ii)] $\mathcal{P}$ is $e$-locally sparse;  
        \item [(iii)] $\mathcal{P}$ is \emph{near-optimal}, i.e., $|\mathcal{P}|\ge(1-\epsilon)\binom{n}{k}/|\mathcal{J}|$. 
    \end{itemize}
\end{lemma}

\section{Cancellative hypergraphs}\label{sec:thm:main-canc}

\subsection{The lower bound}\label{subsec:thm:main-canc-lower}

\noindent The goal of this subsection is to establish the lower bound for the limit in \cref{thm:main-canc}, i.e., for every fixed $\epsilon>0$ there exists $n_0=n_0(t,k,\epsilon)$ such that for all $n\ge n_0$, 
\begin{equation}\label{eq:canc-LB}
    C_{2(t-1)}(n,tk)\ge(1-\epsilon)\cdot \frac{\binom{n}{k}}{\binom{tk-1}{k-1}}.
\end{equation}

Our proof relies on the following construction. Let $\mathcal{F}$ be a fixed $tk$-graph. The \emph{$k$-shadow} of $\mathcal{F}$, denoted by $\sigma_k(\mathcal{F})$, consists of all $k$-subsets of $V(\mathcal{F})$ that are contained in at least one member of $\mathcal{F}$; in other words, $\sigma_k(\mathcal{F})=\cup_{F\in\mathcal{F}}\binom{F}{k}$.

\begin{construction}\label{construction}
Let $\mathcal{F}$ be a fixed $tk$-graph and $\sigma_k(\mathcal{F})$ be its $k$-shadow. Let $\mathcal{P}=\{(V(\mathcal{J}_i),\mathcal{J}_i):i\in[|\mathcal{P}|]\}$ be a \emph{$\sigma_k(\mathcal{F})$-packing} in $\binom{[n]}{k}$, where the $\mathcal{J}_i$'s are pairwise edge-disjoint copies of $\sigma_k(\mathcal{F})$ with vertex set $V(\mathcal{J}_i)$. For each $i$, put a copy $\mathcal{F}_i$ of $\mathcal{F}$ on top of each $\mathcal{J}_i$ such that $\sigma_k(\mathcal{F}_i)=\mathcal{J}_i$ and $V(\mathcal{F}_i)=V(\mathcal{J}_i)$, and let 
    \begin{align}\label{eq:H-sparse1}
        \mathcal{H}(\mathcal{F})=\bigcup_{i=1}^{|\mathcal{P}|} \mathcal{F}_i.   
    \end{align}
\noindent Then $\mathcal{H}(\mathcal{F})\subseteq\binom{[n]}{tk}$ is a $tk$-graph formed by the union of copies of $\mathcal{F}$, and hence $|\mathcal{H}(\mathcal{F})|=|\mathcal{P}|\cdot |\mathcal{F}|$. 
\end{construction}

For a near-optimal packing, this construction gives approximate $\binom{n}{k}|\mathcal{F}|/|\sigma_k(\mathcal F)|$ edges. The next lemma ensures cancellativity. It shows that if $\mathcal{F}$ and $\mathcal{P}$ satisfy the following conditions, then the hypergraph $\mathcal{H}(\mathcal{F})$ defined in \cref{construction} is $2(t-1)$-cancellative.

\begin{lemma}\label{lem:canc-H(F)}
    Let $\mathcal{F}\subseteq\binom{[m]}{tk}$ be a $tk$-graph, and let $\mathcal{P}=\{(V(\mathcal{J}_i),\mathcal{J}_i):i\in[|\mathcal{P}|]\}$ be an induced $\sigma_k(\mathcal{F})$-packing. Suppose that 
    \begin{itemize}
        \item[(i)] $\mathcal{F}$ is $2(t-1)$-cancellative;
        \item[(ii)] $\mathcal{F}$ is $(\ell \cdot tk-(\ell-1)k-1,\ell)$-free for all $2\le \ell\le 2t$; 
        \item[(iii)] $\mathcal{P}$ is $2t$-locally sparse. 
    \end{itemize}
    Then the $tk$-graph $\mathcal{H}(\mathcal{F})$ defined in \cref{construction} is $2(t-1)$-cancellative. 
\end{lemma}
\begin{proof}[Proof of \cref{lem:canc-H(F)}]
    We first show that $\mathcal{H}(\mathcal{F})$ is $(\ell\cdot tk-(\ell-1)k-1,\ell)$-free for every $2\le\ell\le 2t$. Fix such an $\ell$ and let $F_1, \ldots, F_{\ell}$ be any $\ell$ distinct edges of $\mathcal{H}(\mathcal{F})$. Suppose that these edges are distributed among exactly $s$ distinct copies $\mathcal{F}_1,\ldots,\mathcal{F}_s$. For each $i\in[s]$, let $$L_i:=\{j\in[\ell]:F_j\in\mathcal{F}_i\},\quad\ell_i:=|L_i|,\quad X_i:=\bigcup_{j\in L_i}F_j.$$ Then $\ell_i\ge 1$ for every $i\in[s]$ and $\sum_{i=1}^s\ell_i=\ell$. If $\ell_i\ge 2$, then property (ii) gives $|X_i|\ge \ell_i\cdot tk-(\ell_i-1)k$. The same inequality also holds when $\ell_i=1$, since in that case $X_i$ consists of a single $tk$-edge. Hence, for every $i\in[s]$, $|X_i|\ge \ell_i\cdot tk-(\ell_i-1)k$. Since $X_i\subseteq V(\mathcal{F}_i)$, we have $$\left|\bigcup_{i=1}^sV(\mathcal{F}_i)\right|\le\left|\bigcup_{i=1}^sX_i\right|+\left(\sum_{i=1}^s|V(\mathcal{F}_i)|-\sum_{i=1}^s|X_i|\right).$$ 
    Note that each copy $\mathcal{F}_i$ has $m$ vertices. Moreover, since $s\le\ell\le2t$, property (iii) gives $\left|\cup_{i=1}^sV(\mathcal{F}_i)\right|\ge sm-(s-1)k$, when $s\ge2$. The same inequality holds with equality when $s=1$. Therefore, 
    \begin{align*}
        \left|\bigcup_{i=1}^\ell F_i\right|=\left|\bigcup_{i=1}^s X_i\right|
        &\ge \sum_{i=1}^s|X_i|+\left|\bigcup_{i=1}^sV(\mathcal{F}_i)\right|-\sum_{i=1}^s|V(\mathcal{F}_i)|\\
        &\ge \sum_{i=1}^s\left(\ell_i\cdot tk-(\ell_i-1)k\right)+\left(sm-(s-1)k\right)-sm\\
        &=\ell\cdot tk-(\ell-1)k,
    \end{align*}
    Thus every $\ell$ distinct edges of $\mathcal{H}(\mathcal{F})$ have union of size at least $\ell tk-(\ell-1)k$, as required. 
    
    Next, we bound the intersection of an edge with the vertex set of another copy.
    \begin{claim}\label{cla:inducedness}
    Let $\mathcal{F}'$ and $\mathcal{F}''$ be two distinct copies of $\mathcal{F}$ in $\mathcal{H}(\mathcal{F})$. Then for every $F\in \mathcal{F}'$, we have $|F\cap V(\mathcal{F}'')|\le k-1$. 
    \end{claim}
    \noindent Suppose that $|F\cap V(\mathcal{F}'')|\ge k$ and choose a $k$-set $T\subseteq F\cap V(\mathcal F'')$. Since $T\subseteq F\in \mathcal{F}'$, the set $T$ is an edge of the shadow copy $\sigma_k(\mathcal{F}')$. If $|V(\mathcal{F}')\cap V(\mathcal{F}'')|>k$, this contradicts the first condition in the definition of an induced $\sigma_k(\mathcal{F})$ packing. If the intersection has size exactly $k$, then it equals $T$, which is an edge of $\sigma_k(\mathcal{F}')$, contradicting the second condition of an induced $\sigma_k(\mathcal{F})$ packing. Therefore, $|F\cap V(\mathcal{F}'')|\le k-1$. 
    
    Finally, we prove that $\mathcal{H}(\mathcal{F})$ is $2(t-1)$-cancellative.
    Suppose for contradiction that there exist $2t$ distinct edges, say $F_1,\ldots,F_{2t-2},F,F'\in\mathcal{H}(\mathcal{F})$, such that \begin{align}\label{eq:cancellative-not}
        F\Delta F'\subseteq\bigcup_{i=1}^{2t-2}F_i.
    \end{align}
    
    \paragraph{Case 1.} Suppose that $F$ and $F'$ belong to the same copy, say $\mathcal{F}_i$, of $\mathcal{F}$. Since $\mathcal{F}$ is $(2\cdot tk-k-1,2)$-free, we have $|F\cap F'|\le k$. Let $I\subseteq [2t-2]$ be the set of indices $j$ for which $F_j\in \mathcal{F}_i$ and set $|I|=a$. Since $\mathcal{F}$ is $2(t-1)$-cancellative, we have $0\le a<2(t-1)$. Moreover, for each $j\in [2t-2]\setminus I$, the edge $F_j$ lies in a copy of $\mathcal{F}$ different from $\mathcal{F}_i$. Then Claim~\ref{cla:inducedness} gives 
    \begin{align}\label{eq:induced}
        |F_j\cap (F\Delta F')|\le |F_j\cap V(\mathcal{F}_i)|\le k-1.
    \end{align}

    \noindent  Consequently, we have 
    \begin{align*}
        \left|F\cup F'\cup \left(\bigcup_{j\in I} F_j\right)\right|= &\left|\bigcup_{j\in I} F_j\right|+\left|(F\cup F')\setminus\bigcup_{j\in I} F_j\right|\\
        \le&\left|\bigcup_{j\in I} F_j\right|+|F\cap F'|+|F\Delta F'|-\left|(F\Delta F')\bigcap\left(\bigcup_{j\in I}F_j\right)\right|\\
        \le&\left|\bigcup_{j\in I} F_j\right|+|F\cap F'|+\left|(F\Delta F')\bigcap\left(\bigcup_{j\in [2t-2]\setminus I}F_j\right)\right|\\
        \le&a\cdot tk+k+(2t-2-a)\cdot (k-1)\\
        =&(a+2)\cdot tk-(a+1)k-(2t-2-a)\\
        \le&(a+2)\cdot tk-(a+1)k-1,      
    \end{align*}
    
    \noindent where the second inequality follows from \eqref{eq:cancellative-not} and the third inequality follows from \eqref{eq:induced}. This implies that $F,F'$ and $\{F_j:j\in I\}$ form an $((a+2)\cdot tk-(a+1)k-1,a+2)$-configuration with $2\le a+2< 2t$, which contradicts the fact that $\mathcal{H}(\mathcal{F})$ is $(\ell\cdot tk-(\ell-1)k-1,\ell)$-free for all $2\le\ell\le 2t$. 

   \paragraph{Case 2.} Suppose that $F$ and $F'$ belong to different copies of $\mathcal{F}$, say $F\in\mathcal{F}_i$ and $F'\in\mathcal{F}_{i'}$, $i\ne i'$. Then we have $|F\cap F'|\le|F\cap V(\mathcal{F}_{i'})|\le k-1$. Therefore,   
   \begin{align*}
       \left|(F\cup F')\cup\left(\bigcup_{j=1}^{2t-2}F_j\right)\right|=&\left|\bigcup_{j=1}^{2t-2}F_j\right|+\left|(F\cup F')\setminus\bigcup_{j=1}^{2t-2}F_j\right|\le\left|\bigcup_{j=1}^{2t-2}F_j\right|+|F\cap F'|\\
       \le&(2t-2)\cdot tk+(k-1)=2t\cdot tk-(2t-1)k-1,
   \end{align*}
   where the first inequality follows from \eqref{eq:cancellative-not}. This implies that $F,F'$ and $\{F_j:1\le j\le 2t-2\}$ form a $(2t\cdot tk-(2t-1)k-1,2t)$-configuration, which contradicts the fact that $\mathcal{H}(\mathcal{F})$ is $(2t\cdot tk-(2t-1)k-1,2t)$-free.
\end{proof}

We next construct the fixed $tk$-graph $\mathcal F$ used in Construction~\ref{construction}, which satisfies the assumptions (i) and (ii) of Lemma~\ref{lem:canc-H(F)}. 
\begin{lemma}\label{lem:base-canc}
For every $0<\epsilon<1$, there exists an integer $m_0=m_0(t,k,\epsilon)$ such that, for every integer $m\ge m_0$, there exists a nonempty $tk$-graph $\mathcal F\subseteq\binom{[m]}{tk}$ satisfying the following properties:
\begin{itemize}
    \item[(i)] $\mathcal F$ is $s$-cancellative for every integer $s\ge1$;
    \item[(ii)] $\mathcal F$ is $(\ell\cdot tk-(\ell-1)k-1,\ell)$-free for every $2\le\ell\le2t$;
    \item[(iii)]
    $\frac{|\mathcal F|}{|\sigma_k(\mathcal F)|}\ge (1-\epsilon)\frac{1}{\binom{tk-1}{k-1}}.$
\end{itemize}
\end{lemma}

\begin{proof}
    For $2\le \ell\le 2t$, the $(\ell \cdot (tk-1)-(\ell-1)k-1,\ell)$-free condition is automatic whenever $\ell \cdot (tk-1)-(\ell-1)k-1\le tk-1$, so such conditions may be omitted. 
    For every remaining value of $\ell$, $$\frac{\ell\cdot (tk-1)-[\ell\cdot (tk-1)-(\ell-1)k-1]}{\ell-1}=k+\frac{1}{\ell-1}\ge k+\frac{1}{2t-1}.$$
    Applying Lemma~\ref{lem:e^--free} with $r=tk-1$ therefore gives constant $c=c(t,k)>0$ such that, for all sufficiently large $q$, there is a $(tk-1)$-graph $\mathcal G=\{G_1,\dots,G_N\}$ on $[q]$ having $N\ge cq^{k+\frac{1}{2t-1}}$ edges that is $(\ell(tk-1)-(\ell-1)k-1,\ell)$-free for every $2\le\ell\le 2t$. Choose and fix $q$ large enough that $\binom{q}{k}/N\le \epsilon$. 

    Give each $G_i$ a private $u_i$, and set 
    \begin{equation}\label{eq:canc-F}
    \mathcal F:=\{G_i\cup\{u_i\}:i\in[N]\}.
    \end{equation}
    This defines $\mathcal{F}$ on $m_0:=q+N$ vertices. For any $m\ge m_0$, add $m-m_0$ isolated vertices and identify the vertex set with $[m]$. 

    

For property~(i), let $B=G_b\cup\{u_b\}$ and $C=G_c\cup\{u_c\}$ be any two distinguished edges. The vertex $u_b$ belongs to $B\triangle C$ and to no other edge of $\mathcal F$. Therefore $B\triangle C$ cannot be covered by the union of any collection of other edges, and $\mathcal F$ is $s$-cancellative for every $s\ge1$. 

For property~(ii), fix $2\le\ell\le2t$. Without loss of generality, let the chosen $\ell$ edges be $G_1\cup\{u_1\},\ldots,G_\ell\cup\{u_\ell\}$. Then $$\left|\bigcup_{i=1}^{\ell}(G_i\cup\{u_i\})\right|
=\left|\bigcup_{i=1}^{\ell}G_i\right|+\ell
\ge \ell(tk-1)-(\ell-1)k+\ell
=\ell tk-(\ell-1)k,$$
where the inequality uses the corresponding freeness property of $\mathcal G$. 

Finally, every member of $\sigma_k(\mathcal F)$ is either contained in the core $[q]$ or contains the private vertex $u_i$ of one edge $G_i\cup\{u_i\}$. The first type contributes at most $\binom{q}{k}$ members, while for each $i$ the second type contributes at most $\binom{tk-1}{k-1}$ members. Consequently,
$$|\sigma_k(\mathcal F)|\le\binom{q}k+N\binom{tk-1}{k-1},$$
and hence 
\begin{align*}
\frac{|\mathcal F|}{|\sigma_k(\mathcal F)|}\ge\frac{N}{\binom{q}k+N\binom{tk-1}{k-1}}\ge\frac{1}{\epsilon+\binom{tk-1}{k-1}}
\ge(1-\epsilon)\frac{1}{\binom{tk-1}{k-1}},
\end{align*} 
where the last inequality uses $\binom{tk-1}{k-1}\ge1$ and $1/(1+x)\ge1-x$ for $x\ge0$. This proves all three properties for every $m\ge m_0$. 
\end{proof}

We now prove the lower bound in \cref{thm:main-canc}. 
\begin{proof}[Proof of the lower bound in \cref{thm:main-canc}]
We may assume that $0<\epsilon<1$. Apply Lemma~\ref{lem:base-canc} with $\epsilon/2$, fix any $m\ge m_0(t,k,\epsilon/2)$, and let $\mathcal F\subseteq\binom{[m]}{tk}$ be the resulting base hypergraph. Since $\mathcal F$ is nonempty, we have $m\ge tk>k$ and $|\sigma_k(\mathcal F)|\ge\binom{tk}{k}\ge2$, so all hypotheses of Lemma~\ref{lem:conflict-free-matching} are satisfied. 

Using Lemma~\ref{lem:conflict-free-matching} with $\mathcal J=\sigma_k(\mathcal F)$, $e=2t$, and error parameter $\epsilon/2$, for every sufficiently large $n$, there is an induced and $2t$-locally sparse $\sigma_k(\mathcal F)$-packing $\mathcal P$ in $\binom{[n]}k$ such that
\begin{equation}\label{eq:|P|}
|\mathcal P|\ge(1-\epsilon/2)\binom nk\big/|\sigma_k(\mathcal F)|.
\end{equation} 

The freeness of $\mathcal F$ comes from Lemma~\ref{lem:base-canc}(ii), while the local sparsity of $\mathcal P$ comes from Lemma~\ref{lem:conflict-free-matching}(ii). Hence the hypotheses of Lemma~\ref{lem:canc-H(F)} hold, and the hypergraph $\mathcal H(\mathcal F)$ from \cref{construction} is $2(t-1)$-cancellative. Moreover,
\begin{align*}
        |\mathcal{H}(\mathcal{F})|=|\mathcal{P}|\cdot|\mathcal{F}|\ge (1-\epsilon/2)\cdot \frac{\binom{n}{k}}{|\sigma_k(\mathcal{F})|}\cdot |\mathcal{F}|\ge(1-\epsilon/2)^2\frac{\binom nk}{\binom{tk-1}{k-1}}\ge (1-\epsilon)\frac{\binom nk}{\binom{tk-1}{k-1}},
    \end{align*}
    as needed. 
\end{proof}

\subsection{The upper bound}\label{subsec:thm:main-canc-upper}

\noindent The goal of this subsection is to prove the asymptotic upper bound 
$$C_{2(t-1)}(n,tk)\le\binom{n}{k} / \binom{tk-1}{k-1}+O_{t,k}(1),$$ 
which is sufficient for Theorem~\ref{thm:main-canc}. 

We begin by introducing some notation. For a hypergraph $\mathcal{F}$ and $T\subseteq V(\mathcal{F})$, denote by $\deg_{\mathcal{F}}(T)$ the number of edges in $\mathcal{F}$ that contain $T$. Given a subset $X\subseteq V(\mathcal{F})$ and an integer $s\ge 0$, define $\mathcal{D}(\mathcal{F},X,s):=\{T\in\binom{X}{k}:\deg_{\mathcal{F}}(T)=s\}$ and $\mathcal{D}(\mathcal{F},X,\ge s):=\{T\in\binom{X}{k}:\deg_{\mathcal{F}}(T)\ge s\}$. When $X=V(\mathcal{F})$, we simply write $\mathcal{D}(\mathcal{F},s)$ and $\mathcal{D}(\mathcal{F},\ge s)$. With this convention, $|\mathcal{D}(\mathcal F,\ge0)|=\binom nk$.

Let $\mathcal{F}\subseteq\binom{[n]}{tk}$ be a $2(t-1)$-cancellative $tk$-graph. When $|\mathcal F|\ge2t$, we will show that 
\begin{align*}
    \binom{tk}{k}|\mathcal{F}|\le t\binom{n}{k}.
\end{align*}

\noindent To that end, observe that 
\begin{align}\label{eq:canc-upper-1}
    \binom{tk}{k}|\mathcal{F}|=\sum_{s\ge1} s\cdot|\mathcal{D}(\mathcal{F},s)|=\sum_{s\ge1}|\mathcal{D}(\mathcal{F},s)|+\sum_{s\ge2}(s-1)\cdot|\mathcal{D}(\mathcal{F},s)|, 
\end{align} 

\noindent and moreover,
\begin{align}\label{eq:canc-upper-2}
    \sum_{s\ge1}|\mathcal{D}(\mathcal{F},s)|\le \sum_{s\ge0}|\mathcal{D}(\mathcal{F},s)|=|\mathcal{D}(\mathcal{F},\ge 0)|=\binom{n}{k}.
\end{align} 

\noindent Combining \eqref{eq:canc-upper-1} and \eqref{eq:canc-upper-2}, it suffices to show that
\begin{align}\label{eq:canc-upper-3}
    \sum_{s\ge2}(s-1)\cdot|\mathcal{D}(\mathcal{F},s)|\le (t-1)\binom{n}{k}.
\end{align} 

To prove \eqref{eq:canc-upper-3}, we will apply a result on the maximum number of edges in a hypergraph with bounded matching number. Recall that a \emph{matching} in a hypergraph $\mathcal{F}$ is a set of pairwise disjoint edges of $\mathcal{F}$. The \emph{matching number} of $\mathcal{F}$, denoted by $\nu(\mathcal{F})$, is the maximum size of a matching in $\mathcal{F}$. Let $m(n,k,t)$ be the maximum number of edges in an $n$-vertex $k$-graph with matching number at most $t$. Erd\H{o}s \cite{erdos1965problem} posed a well-known conjecture, commonly known as the Erd\H{o}s Matching Conjecture, concerning the exact value of $m(n,k,t)$. This conjecture remains open in general. We refer to \cite{Frankl-Kupavskii-The-Erdos-Matching-Con} for recent progress. For our purpose, we will use the following special case of the Erd\H{o}s Matching Conjecture, proved by  Kleitman \cite{kleitman1968maximal}.

\begin{lemma}[\cite{kleitman1968maximal}, see also \cite{FRANKL80, Frankl-Kupavskii-The-Erdos-Matching-Con}]\label{lem:matching-conjecture}
    $m((t-1)k,k,t-2)=\binom{(t-1)k-1}{k}$ for all positive integers $k\ge2,t\ge 2$.
\end{lemma}

Now we are ready to prove the upper bound in \cref{thm:main-canc}. 

\begin{proof}[Proof of the upper bound in \cref{thm:main-canc}]
    Let $\mathcal{F}\subseteq\binom{[n]}{tk}$ be a $2(t-1)$-cancellative $tk$-graph. If $|\mathcal{F}|<2t$, then $|\mathcal{F}|=O_{t,k}(1)$ and there is nothing to prove for the asymptotic upper bound. Thus we may assume $|\mathcal{F}|\ge 2t$. By the preceding discussion, it suffices to prove \eqref{eq:canc-upper-3}, which in turn yields the desired bound $|\mathcal{F}|\le\binom{n}{k} / \binom{tk-1}{k-1}$.

    Fix $s\ge 2$ and $T\in\mathcal{D}(\mathcal{F},s)$. Let $\{F_1,\ldots,F_s\}$ be the set of edges in $\mathcal{F}$ that contain $T$. For each $i\in[s]$, we consider the family 
    $$\mathcal{D}(\mathcal{F},F_i\setminus T,\ge 2)=\left\{R\in\binom{F_i\setminus T}{k}:\deg_{\mathcal{F}}(R)\ge 2\right\}.$$ The following claim is crucial for the proof.
    \begin{claim}\label{claim:uf-UB}
        There are at least $s-1$ choices of $i\in[s]$ such that $\nu(\mathcal{D}(\mathcal{F},F_i\setminus T,\ge 2))\le t-2$.
    \end{claim}
    Assuming this claim for the moment, let $S=\{i\in[s]:\nu(\mathcal{D}(\mathcal{F},F_i\setminus T,\ge 2))\le t-2\}$. Then $|S|\ge s-1$. For each $i\in S$, \cref{lem:matching-conjecture} yields 
    $$|\mathcal{D}(\mathcal{F},F_i\setminus T,\ge 2)|\le m((t-1)k,k,t-2)=\binom{(t-1)k-1}{k}.$$
    Consequently, for each $i\in S$,
    \begin{align}\label{eq:own-k-subsets}
        |\mathcal{D}(\mathcal{F},F_i\setminus T,1)|\ge\binom{(t-1)k}{k}-\binom{(t-1)k-1}{k}=\binom{(t-1)k-1}{k-1}.
    \end{align}
    
    Define a map \(\sigma:\mathcal{D}(\mathcal{F},\ge 2)\to 2^{\mathcal{D}(\mathcal{F},1)}\) by
    $$\sigma(T):=\{T'\in\mathcal{D}(\mathcal{F},1):\exists F\in \mathcal{F} \text{ such that } T\subseteq F \text{ and } T'\subseteq F\setminus T \}.$$ 
    Since a member of $\mathcal D(\mathcal F,1)$ is contained in a unique edge of $\mathcal F$, the families $\mathcal D(\mathcal F,F_i\setminus T,1)$ are pairwise disjoint. It follows from \eqref{eq:own-k-subsets} that, for each $T\in\mathcal{D}(\mathcal{F},s)$, 
    \begin{align}\label{eq:sigma-T}
        |\sigma(T)|=\sum_{i=1}^s|\mathcal{D}(\mathcal{F},F_i\setminus T,1)|\ge(s-1)\cdot \binom{(t-1)k-1}{k-1}.
    \end{align}
    
    We count the number of pairs $N:=|\{(T,T'):T\in\mathcal{D}(\mathcal{F},\ge2),T'\in\sigma(T)\}|$ in two ways. On the one hand, by \eqref{eq:sigma-T} we have  
    \begin{align}\label{eq:N-lower-bound}
        N=\sum_{T\in\mathcal{D}(\mathcal{F},\ge2)}|\sigma(T)|=\sum_{s\ge2}\sum_{T\in\mathcal{D}(\mathcal{F},s)}|\sigma(T)|\ge \sum_{s\ge2}|\mathcal{D}(\mathcal{F},s)|(s-1)\binom{(t-1)k-1}{k-1}.
    \end{align}
    
    \noindent On the other hand, we may interchange the order of summation. Since each $T'\in \mathcal{D}(\mathcal{F},1)$ lies in exactly one edge $F$, the condition $T'\in \sigma(T)$ implies that $T\subseteq F$ and $T\cap T'=\emptyset$. Hence $T$ must be a $k$-subset of $F\setminus T'$, and therefore there are at most $\binom{(t-1)k}{k}$ possibilities for $T$. Therefore,
    \begin{equation}\label{eq:N-upper-bound}
        \begin{aligned}
        N&=\sum_{T'\in \mathcal{D}(\mathcal{F},1)}|\{T\in \mathcal{D}(\mathcal{F},\ge 2): \exists F\in \mathcal{F} \text{ such that } T\subseteq F \text{ and } T'\subseteq F\setminus T \}|\\
         &\le |\mathcal{D}(\mathcal{F},1)|\cdot\binom{(t-1)k}{k}.
    \end{aligned}
    \end{equation}

    \noindent Combining \eqref{eq:N-lower-bound} and \eqref{eq:N-upper-bound}, we have
    $$\sum_{s\ge2}|\mathcal{D}(\mathcal{F},s)|(s-1)\binom{(t-1)k-1}{k-1}\le |\mathcal{D}(\mathcal{F},1)|\binom{(t-1)k}{k}.$$
    Rearranging, we get 
    $$\sum_{s\ge2}(s-1)\cdot|\mathcal{D}(\mathcal{F},s)|\le (t-1)|\mathcal{D}(\mathcal{F},1)|\le (t-1)\binom{n}{k},$$ as required.
    
    It remains to prove \cref{claim:uf-UB}. Suppose for a contradiction that there are at least two indices $i\in[s]$, say $i=1,2$, for which 
    \begin{align*}
        \nu(\mathcal{D}(\mathcal{F},F_i\setminus T,\ge 2))\ge t-1.
    \end{align*}
    
    \noindent We will derive a contradiction to the assumption that $\mathcal{F}$ is $2(t-1)$-cancellative. 
    
    Since $\nu(\mathcal{D}(\mathcal{F},F_1\setminus T,\ge 2))\ge t-1$, there exist pairwise disjoint sets $T_{1,1},\ldots,T_{1,t-1}\in\mathcal{D}(\mathcal{F},F_1\setminus T,\ge 2)\subseteq\mathcal{D}(\mathcal{F},\ge 2)$. Note that each $T_{1,j}$ has size $k$, and $|F_1\setminus T|=(t-1)k$. Therefore, these sets form a partition of $F_1\setminus T$, and hence $F_1=T\cup T_{1,1}\cup\cdots\cup T_{1,t-1}$. Let 
    $$J_1:=\{j\in[t-1]:T_{1,j}\nsubseteq F_2\}.$$ 
    For each $j\in J_1$, the condition $\deg_{\mathcal F}(T_{1,j})\ge2$ gives an edge $F_{1,j}\in\mathcal F\setminus\{F_1\}$ containing $T_{1,j}$. Because $T_{1,j}\nsubseteq F_2$, this edge is not $F_2$. The sets $T_{1,j}$ with $j\notin J_1$ are already contained in $F_2$, and the sets $T_{1,j}$ partition $F_1\setminus T$.  Hence
    $$F_1\setminus F_2\subseteq\bigcup_{j\in J_1}F_{1,j}.$$
    By symmetry, we obtain a family of at most $t-1$ edges in $\mathcal F\setminus\{F_1,F_2\}$ whose union contains $F_2\setminus F_1$. 
    Combining the two auxiliary families and removing repetitions, we obtain a family of at most $2(t-1)$ distinct edges in $\mathcal F\setminus\{F_1,F_2\}$, whose union contains $F_1\Delta F_2$. Since $|\mathcal F|\ge2t$, this family can be extended, if necessary, to exactly $2(t-1)$ distinct auxiliary edges. Then we obtain two distinct edges $F_1,F_2$ whose symmetric difference is contained in the union of $2(t-1)$ other distinct edges, contradicting the $2(t-1)$-cancellative property. This completes the proof of the claim, and hence the proof of \eqref{eq:canc-upper-3}.  
\end{proof}

\section{Locally $(s,t)$-thin hypergraphs}\label{sec:locally-thin}

\subsection{The lower bound}\label{subsec:hypergraph-lower-bound}
\noindent 
The next lemma gives a sufficient condition for local thinness in terms of two forbidden configurations. 

\begin{lemma}\label{lem:sparse-locally-thin}
    Let $s\ge t\ge1$ and $r\ge2$ be integers, and let $0\le x\le (s-t+1)r-1$ be an integer. If an $r$-graph $\mathcal{H}$ is $((t-1)r+x,s)$-free and $((s-t+1)r-x-1,s-t+1)$-free, then it is locally $(s,t)$-thin. 
\end{lemma}

\begin{proof}
  Suppose, for contradiction, that $A_1,\dots,A_s\subseteq\mathcal{H}$ violate local $(s,t)$-thinness. 
  Choose $s-t+1$ of these edges that have no vertex of degree one in the $s$-edge hypergraph, and relabel them as $A_t,\cdots,A_s$. Let 
  $$X=\bigcup_{i=1}^{t-1} A_i,\qquad Y=\bigcup_{j=t}^s A_j.$$
  In the subhypergraph $\{A_t,\dots,A_s\}$, let $V_1$ be the set of vertices of degree exactly one, and $V_{\ge 2}$ be the set of vertices of degree at least two. Every vertex of $V_1$ lies in $X$, because none of these $s-t+1$ edges has a degree-one vertex in the full $s$-edge hypergraph. The second freeness condition gives $|Y|\ge (s-t+1)r-x$, hence  
  \begin{align}\label{eq:double-count}
   (s-t+1)r=\sum_{u\in Y}|\{t\le i\le s: u\in A_i \}|\ge |V_1|+2|V_{\ge 2}|=|Y|+|V_{\ge 2}|.
  \end{align}
  Thus $|V_{\ge 2}|\le x$, and 
  $$\big|\cup_{i=1}^sA_i\big|=|X|+|Y\setminus X|\le (t-1)r+|V_{\ge 2}|\le (t-1)r+x.$$
  This contradicts the assumption that $\mathcal{H}$ is $((t-1)r+x,s)$-free. Therefore $\mathcal{H}$ is locally $(s,t)$-thin.
\end{proof}

We deduce Theorem \ref{thm:locally-thin-lower} from Lemma \ref{lem:sparse-locally-thin}. 

\begin{proof}[Proof of \cref{thm:locally-thin-lower}]
    If $s=2$, then $t=1$ and $h=r$. Every $r$-graph is locally $(2,1)$-thin, so $L_{(2,1)}(n,r)=\binom nr=\Theta(n^r)$. We may therefore assume that $s\ge3$. 
    
    Choose an admissible integer $x$ attaining the maximum that defines $h$. Consider the two freeness conditions in Lemma~\ref{lem:sparse-locally-thin}: 
    $$\bigl((t-1)r+x,s\bigr)\text{-freeness}\quad\text{and}\quad\bigl((s-t+1)r-x-1,s-t+1\bigr)\text{-freeness}. $$
    Since
    $$\bigl((t-1)r+x\bigr)+\bigl((s-t+1)r-x-1\bigr)=sr-1>2r,$$
    at least one of the two vertex bounds exceeds $r$. Any condition whose vertex bound is at most $r$ holds automatically, as two distinct $r$-edges have a union of size at least $r+1$, and may be omitted. For each remaining condition, the exponent in Lemma~\ref{lem:e^--free} is one of the two terms defining $h$, and hence is at least $h$. Thus Lemma~2.1 yields an $n$-vertex $r$-graph $\Omega(n^{h})$ edges satisfying both freeness conditions. By Lemma~\ref{lem:sparse-locally-thin}, this hypergraph is locally $(s,t)$-thin, proving the result. This completes the proof. 
\end{proof}

\subsection{The upper bound}\label{subsec:hypergraph-upper-bound}
\noindent We first remove edges of degree either zero or at least $s$.  

\begin{lemma}\label{lem:codegree-lower-bound}
    Let $s\ge 1$ and $1\le k\le r$ be integers, and let $\mathcal{H}$ be an $r$-graph on $n$ vertices. Then there exists a subhypergraph $\mathcal{H}'\subseteq \mathcal{H}$ with at least $\max\left\{|\mathcal{H}|-(s-1)\binom{n}{k},\,0\right\}$ edges such that every $k$-set of vertices is contained in either no edge or at least $s$ edges of $\mathcal{H}'$. 
\end{lemma}
\begin{proof}
Begin with $\mathcal{H}$ and iteratively delete edges as follows. As long as there exists a $k$-set that is contained in $\{1,2,\ldots,s-1\}$ edges of the current hypergraph, delete all edges containing that $k$-set.
Since the number of edges decreases at each step, the process must terminate. Let $\mathcal{H}'$ denote the resulting subhypergraph. By construction, no $k$-set is contained in a positive number smaller than $s$ edges of $\mathcal{H}'$. Therefore, every $k$-set of vertices is contained in either no edge or at least $s$ edges of $\mathcal{H}'$.

It remains to estimate the number of edges removed during this procedure. Whenever a $k$-set $S$ triggers deletions, the number of edges containing $S$ at that moment is at most $s-1$. Hence at most $s-1$ edges are deleted. Since there are at most $\binom{n}{k}$ distinct $k$-sets, the total number of deleted edges is at most $(s-1)\binom{n}{k}$. Therefore, $$|\mathcal{H}'|\ge |\mathcal{H}|-(s-1)\binom{n}{k},$$ which completes the proof. 
\end{proof}

Now we turn to the proof of \cref{thm:locally-thin-upper}.

\begin{proof}[Proof of \cref{thm:locally-thin-upper}]
 Let $\mathcal{H}\subseteq\binom{[n]}{r}$ be a locally $(s,t)$-thin $r$-graph. Suppose, for contradiction, that $|\mathcal{H}|\ge (s-1)\binom{n}{\ell}+s$. Then, by \cref{lem:codegree-lower-bound}, there exists a subhypergraph $\mathcal{H}'\subseteq\mathcal{H}$ with $|\mathcal{H}'|\ge s$ such that for every $\ell$-set $T\subseteq[n]$, either $\deg_{\mathcal{H}'}(T)=0$ or $\deg_{\mathcal{H}'}(T)\ge s$. Since locally $(s,t)$-thinness is inherited by subhypergraphs, it suffices to show that $\mathcal{H}'$ is not locally $(s,t)$-thin. Equivalently, we shall find $s$ distinct edges of $\mathcal{H}'$ such that at most $t-1$ of them contain a vertex of degree one in the $s$-edge subhypergraph.
 
 The cases $t=1$ and $t=2$ are immediate. Indeed, in both cases $\left\lfloor \frac{2s-t-1}{s-t+1}\right\rfloor=1$, and hence $\ell=r$; the asserted upper bound then follows from $L_{(s,t)}(n,r)\le\binom{n}{r}$. We may therefore assume that $t\ge3$.
 
 \textbf{Case 1: $r\ge 2\ell$. } Choose an arbitrary edge $B_1\in\mathcal{H}'$ and an $\ell$-subset $T_1\subseteq B_1$. Since $\deg_{\mathcal{H}'}(T_1)\ge s$, we can choose an edge $B_2\in\mathcal{H}'\setminus\{B_1\}$ containing $T_1$. If $s=t+1$, we stop here. If $s\ge t+2$, we continue inductively as follows. Suppose that, for some $2\le i\le s-t$, we have chosen distinct edges $B_1,\ldots,B_i$ and $\ell$-sets $T_1,\ldots,T_{i-1}$ such that
 $$T_j\subseteq B_j\cap B_{j+1}\quad(1\le j\le i-1),\quad T_{j-1}\cap T_j=\emptyset\quad(2\le j\le i-1).$$
 Since $r\ge2\ell$, we may choose an $\ell$-set $T_i\subseteq B_i$ disjoint from $T_{i-1}$. Since $\deg_{\mathcal{H}'}(T_i)\ge s$ and $i\le s-t<s$, fewer than $s$ previously chosen edges are forbidden. Then we can choose a new edge $B_{i+1}\in\mathcal{H}'\setminus\{B_1,\ldots,B_i\}$ containing $T_i$. After $s-t$ steps we have distinct edges $B_1,\ldots,B_{s-t+1}$ and $\ell$-sets $T_1,\ldots,T_{s-t}$ such that
 $$T_i\subseteq B_i\cap B_{i+1}\quad(1\le i\le s-t),\quad T_{i-1}\cap T_i=\emptyset\quad(2\le i\le s-t).$$

 Let $\mathcal{B}=\{B_1,\ldots,B_{s-t+1}\}$. For each $i$, let $C_i$ be the set of vertices of $B_i$ that have degree one in $\mathcal{B}$. Then
 $$C_1\subseteq B_1\setminus T_1,
 \quad C_{s-t+1}\subseteq B_{s-t+1}\setminus T_{s-t},
 \quad C_i\subseteq B_i\setminus(T_{i-1}\cup T_i)\quad(2\le i\le s-t).$$
 Therefore $|C_1|,|C_{s-t+1}|\le r-\ell$, while $|C_i|\le r-2\ell$ for $2\le i\le s-t$.
 Thus $C_1$ and $C_{s-t+1}$ can each be covered by at most $(\lceil r/\ell\rceil-1)$ $\ell$-subsets of $B_1$ and $B_{s-t+1}$, respectively. Similarly, each $C_i$ with $2\le i\le s-t$ can be covered by at most $(\lceil r/\ell\rceil-2)$ $\ell$-subsets of $B_i$. Let $\mathcal{T}$ be the collection of all $\ell$-sets used in these covers. For each $T\in\mathcal{T}$, choose an auxiliary edge $A_T\in\mathcal{H}'$ with $T\subseteq A_T$, greedily and distinctly from all previously chosen edges. This is possible because $\deg_{\mathcal{H}'}(T)\ge s$, and the total number of edges we select is at most
 \begin{align*}
 (s-t+1)+2\left(\left\lc\fr{r}{\ell}\right\rc-1\right)+(s-t-1)\left(\left\lc\fr{r}{\ell}\right\rc-2\right)=(s-t+1)\left\lc\fr{r}{\ell}\right\rc-(s-t-1)\le s.
 \end{align*}
 The last inequality follows from the choice of $\ell$, which gives $$\left\lc\fr{r}{\ell}\right\rc\le \left\lf\fr{2s-t-1}{s-t+1}\right\rf\le \fr{2s-t-1}{s-t+1}.$$ 
 
 Let $\mathcal{F}$ be the subhypergraph consisting of $B_1,\ldots,B_{s-t+1}$ and all auxiliary edges $A_T$. Then $|\mathcal{F}|\le s$. By construction, none of the edges $B_1,\ldots,B_{s-t+1}$ contains a vertex of degree one in $\mathcal{F}$. If $|\mathcal{F}|<s$, add arbitrary edges from $\mathcal{H}'\setminus\mathcal{F}$ until obtaining an $s$-edge subhypergraph. Since $B_1,\ldots,B_{s-t+1}$ still contain no vertex of degree one, this $s$-edge subhypergraph has at most $t-1$ edges containing a vertex of degree one, contradicting local $(s,t)$-thinness.

 \textbf{Case 2: $r<2\ell$.}~ Again choose $B_1\in\mathcal{H}'$ and an $\ell$-subset $T_1\subseteq B_1$, and then choose a distinct edge $B_2$ containing $T_1$. If $s=t+1$, stop. If $s\ge t+2$, continue as follows. Suppose that $B_1,\ldots,B_i$ and $T_1,\ldots,T_{i-1}$ have been chosen for some $2\le i\le s-t$. Since $T_{i-1}\subseteq B_i$ and $r<2\ell$, we have $|B_i\setminus T_{i-1}|=r-\ell<\ell$. Choose an $\ell$-set $T_i\subseteq B_i$ such that $B_i\setminus T_{i-1}\subseteq T_i$. Note that $\deg_{\mathcal{H}'}(T_i)\ge s$ and $i\le s-t<s$. We may choose a new edge $B_{i+1}$ containing $T_i$. Thus we obtain distinct edges $B_1,\ldots,B_{s-t+1}$ such that
 $$B_i\subseteq T_{i-1}\cup T_i\subseteq B_{i-1}\cup B_{i+1}\qquad(2\le i\le s-t).$$
 Hence none of the internal edges $B_2,\ldots,B_{s-t}$ has a vertex of degree one in $\{B_1,\ldots,B_{s-t+1}\}$. Let $C_1$ and $C_{s-t+1}$ be the sets of degree-one vertices in the edges $B_1$ and $B_{s-t+1}$, respectively. We have
 $$C_1\subseteq B_1\setminus T_1,
 \quad C_{s-t+1}\subseteq B_{s-t+1}\setminus T_{s-t},$$
 so $|C_1|,|C_{s-t+1}|<\ell$. Choose $\ell$-sets $S_1\subseteq B_1$ and $S_{s-t+1}\subseteq B_{s-t+1}$ with $C_1\subseteq S_1$ and $C_{s-t+1}\subseteq S_{s-t+1}$. By the same greedy argument, choose distinct edges $A_1,A_{s-t+1}\in\mathcal{H}'\setminus\{B_1,\ldots,B_{s-t+1}\}$ with $S_1\subseteq A_1$ and $S_{s-t+1}\subseteq A_{s-t+1}$, also requiring $A_1\neq A_{s-t+1}$. The resulting subhypergraph $\{B_1,\ldots,B_{s-t+1},A_1, A_{s-t+1}\}$ has $s-t+3\le s$ edges, and none of $B_1,\ldots,B_{s-t+1}$ has a vertex of degree one. Extending it to an $s$-edge subhypergraph of $\mathcal{H}'$ again contradicts local $(s,t)$-thinness. This contradiction proves the desired upper bound.
\end{proof}

Next we complete the proof of Corollary~\ref{cor:locally-thin}. 
\begin{proof}[Proof of Corollary~\ref{cor:locally-thin}]
The assumption gives $\ell=r/d$ in Theorem~\ref{thm:locally-thin-upper}, so $L_{(s,t)}(n,r)=O(n^{r/d})$. For the matching lower bound, choose $x:=(s-t)r/d$.
This is a nonnegative integer, and $(s-t+1)r-x=(2s-t-1)r/d-(s-t)r/d=(s-1)r/d\ge 1.$
Thus $0\le x\le(s-t+1)r-1$, so this value of $x$ is admissible in
Theorem~\ref{thm:locally-thin-lower}. For this choice, $\frac{(s-t+1)r-x}{s-1}=r/d$ and $\frac{x+1}{s-t}=r/d+\frac{1}{s-t}>r/d$. 
It follows that the exponent $h$ in Theorem~\ref{thm:locally-thin-lower} satisfies $h\ge r/d$, and hence
$L_{(s,t)}(n,r)=\Omega(n^{r/d})$. Together with the upper bound, this gives
$$L_{(s,t)}(n,r)=\Theta(n^{r/d})=\Theta\!\left(n^{(s-t+1)r/(2s-t-1)}\right),$$
as required. 
\end{proof}

\section*{Acknowledgements}
\noindent We would like to thank Mengfei Tang for helpful discussions at the early stage of this project. M. Liu, C. Shangguan, and C. Zhang are supported by the National Natural Science Foundation of China under Grant Nos. 12571352 and 12231014, and the Fundamental Research Funds for the Central Universities. 

All mathematical ideas and proofs presented in this paper are developed and written by the authors. ChatGPT-5.6 Sol is used to improve the language and presentation of the manuscript and to check the completeness and consistency of the references. 

{\small
\normalem
\bibliographystyle{plain}
\bibliography{ref}
}

\end{document}